\documentclass[12pt,reqno]{amsart}
\usepackage{amssymb}

\newcommand{\e}{\varepsilon}

\newcommand{\lam}{\lambda}

\newcommand{\fy}{\varphi}

\newcommand{\p}{\partial}
\newcommand{\I}{\infty}

\newcommand{\R}{\mathbb{R}}

\newcommand{\Z}{\mathbb{Z}}

\newcommand{\calS}{\mathcal{S}}
\newcommand{\rad}{{\operatorname{radial}}}

\newcommand{\Sg}{\mathfrak{S}}

\newcommand{\calL}{\mathcal{L}}
\newcommand{\mm}{\mathfrak{m}}

\numberwithin{equation}{section}

\newtheorem{thm}{Theorem}[section]
\newtheorem{cor}[thm]{Corollary}
\newtheorem{lem}[thm]{Lemma}
\newtheorem{prop}[thm]{Proposition}

\theoremstyle{remark}

\newcommand{\Ext}{E_{\operatorname{ext}}}

\newcommand{\ran}{\rangle}
\newcommand{\lan}{\langle}

\newcommand{\lec}{\lesssim}
\newcommand{\gec}{\gtrsim}

\newcommand{\EQ}[1]{\begin{equation}  \begin{split} #1 \end{split} \end{equation} }
\newcommand{\CAS}[1]{\begin{cases} #1 \end{cases}}
\newcommand{\pt}{&}
\newcommand{\pr}{\\ &}
\newcommand{\pq}{\quad}
\newcommand{\pn}{}

\newcommand{\LR}[1]{{\lan #1 \ran}}
\newcommand{\de}{\delta}

\newcommand{\ta}{\tau}

\newcommand{\ka}{\kappa}
\newcommand{\ga}{\gamma}

\newcommand{\s}{\sigma}
\newcommand{\rh}{\rho}
\newcommand{\na}{\nabla}

\newcommand{\supp}{\operatorname{supp}}

\newcommand{\B}{\mathcal{B}}
\newcommand{\De}{\Delta}

\newcommand{\IN}[1]{\text{ in }#1}

\newcommand{\sg}{\mathfrak{s}}
\newcommand{\HH}{\mathcal{H}}
\newcommand{\V}{\vec}
\newcommand{\La}{\Lambda}

\newcommand{\dist}{\mathrm{dist}}

\newcommand{\eps}{\varepsilon}
\newcommand{\sign}{\mathrm{sign}}

\newcommand{\spec}{\mathrm{spec}}

\begin{document}

\author{J.~Krieger}
\address{Department of Mathematics, The University of Pennsylvania, 209 South 33rd Street, Philadelphia, PA 19104, U.S.A.}

\author{K.~Nakanishi}
\address{Department of Mathematics, Kyoto University, Kyoto 606-8502, Japan}

\author{W.~Schlag}
\address{Department of Mathematics, The University of Chicago, 5734 South University Avenue, Chicago, IL 60615, U.S.A.}

\thanks{The first author wishes to thank the Mathematics Department of the University of Chicago for its hospitality in July 2010, when part of this work was done.  Support of the National Science Foundation, DMS-0757278 (JK) and  DMS-0617854 (WS)  as well as the Sloan Foundation (JK) is gratefully acknowledged. The third author was partially supported by a Guggenheim fellowship.}

\title[Global dynamics for the critical wave equation]{Global dynamics away from the ground state \\ for the energy-critical nonlinear wave equation}

\begin{abstract}
We study global behavior of radial solutions for the nonlinear wave equation with the focusing energy critical nonlinearity in three and five space dimensions. 
Assuming that the solution has energy at most slightly more than the ground states and gets away from them in the energy space, we can classify its behavior into four cases, according to whether it blows up in finite time or scatters to zero, in forward or backward time direction. 
We prove that initial data for each case constitute a non-empty open set in the energy space. 

This is an extension of the recent results \cite{NakS,NakS2} by the latter two authors on the subcritical nonlinear Klein-Gordon and Schr\"odinger equations, except for the part of the center manifolds. 
The key step is to prove the ``one-pass" theorem, which states that the transition from the scattering region to the blow-up region can take place at most once along each trajectory. 
The main new ingredients are the control of the scaling parameter and the blow-up characterization by Duyckaerts, Kenig and Merle \cite{DKM1,DKM2}. 
\end{abstract}

\maketitle 

\tableofcontents

\section{Introduction}
We consider the $H^1$-critical, focusing  nonlinear wave equation 
\EQ{ \label{eq:H1critical}
 \ddot u - \De u = |u|^{2^*-2}u,\pq u(t,\,x):\R^{1+d} \to \R, \pq 2^*=\frac{2d}{d-2}\pq (d= 3 \text{ or }5), 
}
in the radial context, where $2^*$ denotes the $H^1$ Sobolev critical exponent. We remark that the dimensional restriction is needed only for using the  blow-up characterization by Duyckaerts-Kenig-Merle~\cite{DKM2}. 

We take the radial energy space as the phase space for the above equation, which can be normalized to $L^2$ by putting 
\EQ{
 \vec u:=(|\na| u,\dot u) \in L^2_\rad(\R^d)^2=:\HH,}
at each time $t\in\R$, where $|\na|=\sqrt{-\De}$ is an isometry from $\dot H^1_\rad(\R^d)$ onto $L^2_\rad(\R^d)$. 
Thus, to any scalar space-time function $u(t,x)$, we will associate the vector function $\vec u(t,x)$ by the above relation. Conversely, for any time independent $\vec\fy=(\fy_1,\fy_2)\in\HH$, we introduce the following notation 
\EQ{
 \fy:=|\na|^{-1}\fy_1, \pq \dot\fy:=\fy_2.}

The conserved energy of \eqref{eq:H1critical} is denoted by 
\EQ{
 E(\vec u):=\int_{\R^d} \Bigl[\frac{|\dot u|^2+|\na u|^2}{2}-\frac{|u|^{2^*}}{2^*}\Bigr] dx.}
It is well-known that this problem admits the static Aubin solutions of the form 
\EQ{
 W_{\lambda} = T_\lam W,\pq W(x) = \left[1+\frac{|x|^2}{d(d-2)}\right]^{1-\frac{d}{2}},}
where $T_\lam$ denotes the $\dot H^1$ preserving dilation 
\EQ{
 T_\lam\fy=\lam^{d/2-1}\fy(\lam x).}
These are positive radial solutions of the static equation
\EQ{
 -\Delta W -|W|^{2^*-2}W=0,} 
which are unique, up to dilation and translation symmetries, amongst the non-negative, non-zero (not necessarily radial) $C^{2}$ solutions, see~\cite{CGS}. They also minimize the static energy 
\EQ{
 J(\fy) := \int_{\R^{d}} \Bigl[ \frac12 |\nabla \fy|^{2} - \frac{1}{2^*} |\fy|^{{2^*}} \Bigr] dx,}
among all non-trivial static solutions. 
The work of Kenig, Merle~\cite{KM1,KM2} and Duyckaerts, Merle~\cite{DM1,DM2} allows for a characterization of the global-in-time behavior of solutions with $E(\vec u)\leq J(W)$. 

In this paper we  study the behavior of solutions with 
\EQ{ \label{eq:energycond}
 E(\vec u)<J(W)+\eps_0^{2},
}
for some small $\eps_0>0$. 
Solutions of subcritical focusing NLKG and NLS equations with radial data in~$\R^{3}$ of energy slightly above that of the ground state were studied by the latter two authors in~\cite{NakS,NakS2}. 
Our goal in this paper is to extend those results to the critical case. 
The key feature of \eqref{eq:H1critical} by contrast to NLKG is the scaling invariance of~\eqref{eq:H1critical} manifested by 
\EQ{ 
 u(t, x) \mapsto \lambda^{\frac{d}{2}-1}u(\lambda t, \lambda x) = T_\lam u(\lam t) 
}
which leaves the energy unchanged. 
In particular, the analogue of the ``one pass theorem'' proved in \cite{NakS} needs to be modified, specifically by replacing the discrete set of attractors $\{Q, -Q\}$ there by the one--parameter family of the ground states 
\EQ{
 \calS: = \{W_\lambda\}_{\lambda>0}.}
Note that in the subcritical NLS case \cite{NakS2}, the scaling parameter $\lam$ (in frequency) is essentially fixed or at least bounded from above and below by the $L^2$ conservation law, but in the critical case there is no factor which a priori prevents the scale from going to $0$ or $+\I$. 

Introduce the ``virial functional''
\EQ{
 K(\fy): = \int_{\R^d}[|\nabla\fy|^2 - |\fy|^{2^*}]\,dx} 
 and note that $K(W)=0$. 
The following positivity is crucial for the variational structure around $W$
\EQ{
 H(\fy) := \|\na\fy\|_2^2/d = J(\fy)-K(\fy)/2^*.}
Note that the derivative of $J(\fy)$ with respect to any scaling $\fy(x)\mapsto \lam^a\fy(\lam^bx)$ except for $T_\lam$ gives a non-zero constant multiple of $K(\fy)$. This is a special feature of the scaling critical case, which allows us to work with a single $K$, whereas in the subcritical case \cite{NakS} we needed two different functionals and their equivalence.  
 
The main result of this paper is summarized as follows. 
\begin{thm}\label{thm: Main} 
There exist a small $\e_*>0$, a neighborhood $\B$ of $\V\calS$ within $O(\e_*)$ distance in $\HH$, and a continuous functional 
\EQ{
 \Sg:\{\V\fy\in\HH\setminus\B \mid E(\V\fy)<J(W)+\e_*^2\} \to \{\pm 1\},}such that the following properties hold: For any solution $u$ with $E(\V u)<J(W)+\e_*^2$ on the maximal existence interval $I(u)$, let 
\EQ{
 \pt I_0(u):=\{t\in I(u)\mid \V u(t)\in \B\}, 
 \pr I_\pm(u) := \{t\in I(u) \mid \V u(t)\not\in\B,\ \Sg(\V u(t))=\pm 1\}.}
Then $I_0(u)$ is an interval, $I_+(u)$ consists of at most two infinite intervals, and $I_-(u)$ consists of at most two finite intervals. $u(t)$ scatters to $0$ as $t\to\pm\I$ if and only if $\pm t\in I_+(u)$ for large $t>0$. Moreover, there is a uniform bound $M<\I$ such that 
\EQ{
 \|u\|_{L^q_{t,x}(I_+(u)\times\R^d)}\le M, \pq q:=\frac{2(d+1)}{d-2}.}
For each $\s_1,\s_2\in\{\pm\}$, let $A_{\s_1,\s_2}$ be the collection of initial data $\V u(0)\in\HH$ such that $E(\V u)<J(W)+\e_*^2$, and for some $T_-<0<T_+$, 
\EQ{
 (-\I,T_-)\cap I(u) \subset I_{\s_1}(u), \pq (T_+,\I)\cap I(u) \subset I_{\s_2}(u).}
Then each of the four sets $A_{\pm,\pm}$ is open and non-empty, exhibiting all possible combinations of scattering to zero/finite time blowup as $t\to\pm\I$, respectively. 
\end{thm}

The neighborhood $\B$ as well as the sign functional $\Sg$ will be defined explicitly, cf.~Corollary \ref{cor: onepass}. 
In short, every solution $u$ with energy $E(\V u)<J(W)+\e_*^2$ can change the sign $\Sg(\V u(t))$ at most once, by entering the neighborhood $\B$, 
and $u$ scatters/blows-up if it keeps $\Sg=+1$ / $-1$. 
This is the same description as in the subcritical case \cite{NakS} concerning the dynamics away from the ground states $\calS$. 
Indeed, the part about the sign change seems fairly general, which we called ``one-pass" theorem, relying only on the energy
 and virial type arguments. It will be proved separately as the first step in Theorem \ref{thm: onepass}. 
 
However, we do not know at this time how to deal with solutions $u$ which stay in $\B$.  
Note that the solutions constructed in~\cite{KST} in the three-dimensional case belong to this tube. 
Moreover, Duyckaerts, Kenig, Merle~\cite{DKM1} showed that all type-II blowup (i.e.,~blowup with bounded energy norm) under the constraint~\eqref{eq:energycond} is of the form of those solutions found in~\cite{KST}. 
But the tube around~$\calS$ might also contain solutions which do not blow up but rather scatter to~$\calS$. 
This would correspond to the center-stable manifolds in~\cite{NakS,NakS2}. 
However, in contrast to~\cite{NakS,NakS2} we do not address the issue of existence of a center-stable manifold associated with~\eqref{eq:H1critical},
nor do we give a complete description of all possible dynamics for solutions as in~\eqref{eq:energycond}. 
Recall that~\cite{KS} establishes the existence of such a manifold for the radial three-dimensional critical wave equation, but not in the energy
class. It appears to be a delicate question in any dimension to decide whether or not a center-stable manifold associated with the ground states exists 
in the case of  energy critical equations.

The key idea behind the proof is similar to the one in~\cite{NakS}, which relies on an interplay between the hyperbolic dynamics of the linearized operator around $W$ with the variational structure of $J$ and $K$ away from $\calS$. 

Dynamically speaking, the linearization around $W$ is delicate, as one needs to take a time-dependent scaling parameter $\lam(t)$ into account. This is a major difference from~\cite{NakS}. 
To address it, we use the observation that the evolution of $\lam(t)$ is much slower than that of the exponentially unstable mode. 
Indeed, the evolution of $\lambda(t)$ is governed by the threshold eigenvalue (which lies at zero energy) of the linearized operator and is therefore by nature algebraically unstable rather than exponentially unstable. 
This will allow us to freeze the dilation parameter in those time intervals during which the trajectories are dominated by the hyperbolic (and unstable) dynamics. 

The other major difference, which could be more serious, is the possibility of concentration blow-up in the region $K\ge 0$ and away from $\calS$, where the solutions are bounded and so automatically global in the subcritical case. 
This problem arises after applying the one-pass theorem. 
Fortunately, we will see that the blow-up analysis by Duyckaerts, Merle, Kenig \cite{DKM1,DKM2} precludes it, so that we can proceed essentially in the same way as in the subcritical case. 
 
\section{Energy distance functional} 

In this section we define the nonlinear distance functional to the ground state family $\calS$, by using the linearized operator, but still keeping the nonlinear structure, so that it will best reflect the hyperbolic nature around $\calS$. 
The main difference from the subcritical radial NLKG \cite{NakS} is that we need a good choice of the scaling parameter.  

Let $\rh>0$ be the unique $L^{2}$-normalized ground-state for the linearized operator \EQ{
 \calL := -\De - (2^*-1)W^{2^*-2}, \pq \calL\rh=-k^2\rh, \pq \|\rh\|_2=1.} 
Then $\rh_\lam:=T_\lam\rh$ is a ground state of the rescaled operator 
\EQ{
 \calL_\lam := -\De - (2^*-1)W^{2^*-2}_\lam, \pq \calL_\lam\rh_\lam=-k^2\lam^2\rh_\lam, \pq \|\rh_\lam\|_2=1/\lam.}

Expand $u$ around $W_\lam$ by 
\EQ{
 u = W_\lam + v_\lam \pt= W_\lam + \mu_\lam(u)\rh_\lam + \ga_\lam, \pq \ga_\lam\perp\rh_\lam,}
where $\mu_\lam$ is given by 
\EQ{
 \mu_\lam(\fy) := \LR{\fy-W_\lam|\lam^2\rh_\lam}=\LR{\fy-W_\lam|T_{1/\lam}^*\rh}.}Since $\rho\in\calS\subset\dot H^{-1}$, we obtain by rescaling 
\EQ{ 
 |\mu_\lam(u)|\lesssim \|\nabla v_\lam\|_{L^2}.} 
Note that $\ga_\lam$ may contain the root mode in the direction 
\EQ{
 \p_\lam W_\lam = \La W_\lam, \pq \La:=r\p_r+d/2-1.}
 However, this will not cause any problems in our analysis of  
 the hyperbolic dynamics. The energy is expanded as 
\EQ{ \label{energy expand}
 E(\vec u)-J(W) &= \frac 12[\|\dot u\|_2^2 + \LR{\calL_\lam v_\lam| v_\lam}]-C_\lam(v)\\
&= \frac 12[\|\dot u\|_2^2 - k^2 \mu_\lam(\fy)^2 + \LR{\calL_\lam \ga_\lam| \ga_\lam}] -C_\lam(v_\lam),
}
where $C_\lam$ denotes the superquadratic part of the energy, i.e., 
\EQ{
 C_\lam(v) \pt:= \int_{\R^d}\Bigl[\frac{|v+W_\lam|^{2^*}-|W_\lam|^{2^*}}{2^*}-W_\lam^{2^*-1}v-\frac{2^*-1}{2}W_\lam^{2^*-2}|v|^2\Bigr] dx 
 \pr=O(\|v\|_{\dot{H}^1}^3).} 
In the same way as in \cite{NakS}, we introduce an energy functional
\EQ{
 E_\lam(\vec u) \pt:=E(\vec u)-J(W)+k^2\mu_\lam(u)^2
 \pr=\frac 12[\|\dot u\|_2^2 + k^2 \mu_\lam(\fy)^2 + \LR{\calL_\lam \ga_\lam| \ga_\lam}] -C_\lam(v_\lam).}

Now we choose $\lam=\lam(u)$ for $u$ close to $\calS$ by the orthogonality condition 
\EQ{ \label{lam orth}
 \LR{u|\La^*\rh_\lam}=\LR{v_\lam|\La^*\rh_\lam}=0,}
using the fact that 
\EQ{
 \LR{W_\lam|\La^*\rh_\lam}=\LR{\La W_\lam|\rh_\lam}=0,} 
which follows from $\calL_\lam \La W_\lam=0$ and $\calL_\lam\rh_\lam=-k^2\lam^2\rh_\lam$, and $\rh\in\calS$. Such $\lam(u)$ is uniquely determined at least in the region 
\EQ{
 \|\na v_\lam\|_2 \sim \dist_{\dot H^1}(u,\calS)\ll 1,}
by the implicit function theorem, since 
\EQ{ \label{nondege orth}
 \p_{\lam=1} \LR{W|\La^*\rh_\lam}\pt=\LR{\La W|\La^*\rh}=\LR{\La W|-k^{-2}(\La(2^*-2)W^{2^*-2})\rh}
 \pr=-k^{-2}(2^*-1)(2^*-2)\LR{W^{2^*-3}(\La W)^2|\rh}<0.} 

In order to bound the remainder by the energy, we use the following result. 

\begin{lem} \label{lem:non-neg spec}
For any $\ga\in\dot H^1_{rad}$ such that $\ga\perp\rh$, we have $\LR{\calL\ga|\ga}\ge 0$ and 
\EQ{
 \|\na\ga\|_2^2 \sim \LR{\ga|\La^*\rh}^2+\LR{\calL\ga|\ga}.}
\end{lem}
\begin{proof}
Let $z=|\na|\ga\in L^2$, then the bilinear form is rewritten 
\EQ{
 \LR{\calL\ga|\ga}=\LR{(1-A)z|z}, \pq A:=|\na|^{-1}W^{2^*-2}|\na|^{-1}.}
$A$ is a positive, compact and self-adjoint operator on $L^2$. Hence $\spec(1-A)$ is bounded, discrete, with the only accumulation point being~$1$. Since
\EQ{
 z \perp |\na|^{-1}\rh \perp |\na|\La W \in (1-A)^{-1}(0),}
we can decompose 
\EQ{
 \pt z = c|\na|\La W + z_+, \pq z_+\perp\{|\na|^{-1}\rh,|\na|\La W\},}
then  
\EQ{
  \LR{\calL\ga|\ga}=\LR{(1-A)z_+|z_+}.}
First we prove 
\EQ{ \label{1-A pos bd}
 L^2_\rad \ni z\perp\{|\na|^{-1}\rh,|\na|\La W\} \implies \LR{(1-A)z|z}\sim\|z\|_2^2,}
noting that 
\EQ{ \label{L pos}
 L^2\ni z\perp |\na|^{-1}\rh \implies \LR{(1-A)z|z}\ge 0.}
Suppose \eqref{1-A pos bd} fails. Then there exists a sequence $z_n\in L^2_\rad$ such that $\|z_n\|_2=1$, $z_n\to z_\I$ weakly and $\LR{(1-A)z_n|z_n}\to 0$ as $n\to\I$. Since $Az_n\to Az_\I$ strongly, we have 
\EQ{
 \LR{(1-A)z_\I|z_\I}\le 0, \pq z_\I \perp |\na|^{-1}\rh, |\na|\La W.}
Then \eqref{L pos} implies that $\LR{(1-A)z_\I|z_\I}=0$, and $z_n\to z_\I$ strongly. So there is a Lagrange multiplier $c\in\R$ such that 
\EQ{
 (1-A)z_\I = c|\na|^{-1}\rh.}
On the other hand, $\calL\rh=-k^2\rh$ gives $(1-A)|\na|\rh=-k^2|\na|^{-1}\rh$, whence 
\EQ{
 c=\LR{(1-A)z_\I||\na|\rh}=\LR{z_\I|(1-A)|\na|\rh}=0,}
and thus 
\EQ{
 (1-A)z_\I=0.}
This implies that for some $b\in\R$,
\EQ{
 z_\I = b |\na|\La W.}
Since $z_\I\perp |\na|\La W$, we conclude that $z_\I=0$, which contradicts the strong convergence and $\|z_n\|=1$. Thus \eqref{1-A pos bd} is proved. 

It remains to bound $c$. Since 
\EQ{
 \LR{\ga|\La^*\rh}=c\LR{\La W|\La^*\rh}+\LR{z_+||\na|^{-1}\La^*\rh},}
$\LR{\La W|\La^*\rh}<0$ by \eqref{nondege orth} and $|\na|^{-1}\La^*\rh\in L^2$, we infer that
\EQ{
 |c| \lec |\LR{\ga|\La^*\rh}|+\|z_+\|_2,}
which together with \eqref{1-A pos bd} implies the desired estimate.  
\end{proof}

Thus we deduce that, if $u$ is close enough to $\calS$ and $\lam=\lam(u)$ then  
\EQ{ \label{energy equiv}
 E_{\lam(u)}(\vec u) \pt\sim \|\dot u\|_2^2 + |\mu_\lam(u)|^2 + \|\na \ga_\lam\|_2^2 + O(\|\na v_\lam\|_2^3)
 \pr\sim \|\dot u\|_2^2+\|\na v_\lam\|_2^2 \sim \|\vec u-(|\na|W_{\lam(u)},0)\|_2^2. }
For brevity, we write 
\EQ{ \label{def muS}
 \mu_\calS(u):=\mu_{\lam(u)}(u), \pq E_\calS(\vec u):=E_{\lam(u)}(\vec u),}
when $u$ is close to $\calS$. 

Now we can define our distance function $d_\calS(\vec\fy)$. 
Let $\chi(r)\in C_0^\I(\R)$ be a symmetric decreasing function such that 
\EQ{ \label{def chi}
 \chi(r)=\CAS{1 &(|r|\le 1) \\ 0 &(|r|\ge 2).}} 
Let $d_0$ denote the linear distance from $\calS$ 
\EQ{
 d_0(\vec\fy):=\inf_{\nu>0}\|\vec\fy-\vec W_\nu\|_2,}
and then define 
\EQ{ \label{def d_S}
 d_\calS(\vec\fy):=\pt\chi(d_0(\vec\fy)/\de_E)E_\calS(\vec\fy)^{1/2} 
 +\chi(d_0(-\vec\fy)/\de_E)E_\calS(-\vec\fy)^{1/2}
 \pr+[1-\chi(d_0(\vec\fy)/\de_E)-\chi(d_0(-\vec\fy)/\de_E)]C_E\min_\pm d_0(\pm\vec\fy),}
for some fixed $0<\de_E\ll \min(1,\|\na W\|_2)$ and $C_E\gg 1+\|\na W\|_2$, such that for $d_0(\vec\fy)<2\de_E$, $\vec\fy$ is close to either $\vec \calS=\{(|\na|W_\lam,0)\}_{\lam>0}$ or $-\vec \calS$, and 
\EQ{
 d_\calS(\fy)^2 = E_\calS(\pm\vec\fy) = E(\vec\fy)-J(W)+k^2\mu_\calS(\pm\fy)^2.}
Since $\lam(\fy)$, $\mu_\calS$ and $E_\calS$ have been defined only near $\calS$, it is harmless and convenient to extend them evenly around $-\calS$:
\EQ{
 E_\calS(\vec\fy):=E_\calS(-\vec\fy), \pq \mu_\calS(\fy):=\mu_\calS(-\fy), \pq \lam(\fy):=\lam(-\fy).}
Thus $d_\calS:\HH\to[0,\I)$ is continuous and even, satisfying 
\EQ{
 d_\calS(\vec\fy) \sim \min_\pm d_0(\vec\fy)=\dist_{L^2}(\vec\fy,\vec \calS\cup-\vec \calS).}
The following lemma shows the basic property of the distance: once we are slightly away from $\calS$, then the unstable mode $\mu_\calS(u)$ becomes the dominant part of the distance. Our analysis in this paper is mostly in this region.

\begin{lem} \label{lem:eigendom}
For any $\vec\fy\in\HH$ satisfying 
\EQ{
 E(\vec\fy)-J(W) \le d_\calS(\vec\fy)^2/2,  \pq d_\calS(\vec\fy) \le \de_E,}
one has  $|\mu_\calS(\fy)| \sim d_\calS(\vec\fy) = E_\calS(\fy)^{1/2}$. 
\end{lem}
\begin{proof}
By definition of $d_\calS$, we have 
\EQ{
 d_\calS(\vec\fy)^2=E_\calS(\vec\fy)\pt=E(\vec\fy)-J(W)+k^2|\mu_\calS(\fy)|^2,}
and so $d_\calS(\vec\fy)^2-k^2|\mu_\calS(\fy)|^2 < d_\calS(\vec\fy)^2/2$, 
which implies $|\mu_\calS(\fy)| \gec d_\calS(\vec\fy)$, while the other direction of the inequality is always true by \eqref{energy equiv}. 
\end{proof}

\section{Variational structure} 
In this section, we prove the following crucial variational type lemma, which is used to control the dynamics away from the ground states in the proof of the ``one pass theorem''. 
 Here the argument is static in the phase space $\HH$. 
Due to the underlying scaling invariance, we need to use the concentration compactness approach. 
\begin{lem}\label{lem: variational} 
There is a continuous increasing function $\eps_V:(0,\I)\to(0,1)$ such that if $\vec\fy\in\HH$, $E(\vec\fy)\le J(W)+\eps_V(\de)^{2}$ and $d_\calS(\vec\fy)\ge\de$ for some $\de>0$, then we have either 
\EQ{
 K(\fy) \ge \min\{\kappa(\delta),\,c\|\nabla\fy\|_{L^2}^2\}
}
 or else 
\EQ{
 K(\fy) \le -\kappa(\delta)
}
for suitable $\kappa(\delta)>0$ and an absolute constant $c>0$. 
\end{lem}
\begin{proof} 
We may assume $\e_V(\de)\ll\de\ll\de_E$. If $\|\dot\fy\|_2\ll\de$, then we have $\de<d_\calS(\vec\fy) \sim \dist_{\dot H^1}(\fy,\calS\cup-\calS)$. 
Otherwise, $J(u)<J(W)-O(\de^2)$ and so $\dist_{\dot H^1}(\fy,\calS\cup-\calS)\gec\de^2$. 
The conclusion is clear for $\|\na\fy\|_2\ll 1$ by Sobolev. Hence,
 if the conclusion fails for some $\de>0$, then there exists a sequence $\fy_n\in\dot H^1_\rad$ such that $\|\na\fy_n\|_2\gec 1$ and 
\EQ{
 J(\fy_n)<J(W)+1/n,\ |K(\fy_n)|<1/n,\ \dist_{\dot H^1}(\fy_n,\calS)\gec\de^2.}
The first two conditions together with $K(W)=0$ imply that 
\EQ{
 \limsup_{n\to\I}H(\fy_n) \le H(W),}
and so $\fy_n$ is bounded in $\dot H^1\subset L^{2^*}\cap rL^2$ by the Sobolev and Hardy inequalities. 
We deal with possible concentration by the dyadic decomposition in $x\in\R^d$:
\EQ{
 D_j^<:=\{|x|<2^j\},\pq D_j:=\{2^j<|x|<2^{j+1}\}, \pq D_j^>:=\{2^{j+1}<|x|\}.}
First we show that for any $\e>0$, there is $\nu>0$ such that for any $h\in\Z$ and $n$, 
\EQ{ \label{no dichotomy}
 \|\fy_n/r\|_{L^2(D_h^<)}>\e,\ \|\fy_n/r\|_{L^2(D_h^>)}>\e \implies \|\fy_n/r\|_{L^2(D_h)}>\nu.}
If this fails for some $\e>0$, then along a subsequence there exist $h_n$ such that 
\EQ{
 \|\fy_n/r\|_{L^2(D^<_{h_n})}>\e,\ \|\fy_n/r\|_{L^2(D^>_{h_n})}>\e, \pq \|\fy_n/r\|_{L^2(D_{h_n})}\to 0.}
Let $\fy_n^0:=\chi(2^{-h_n}|x|)\fy_n$ and $\fy_n^1:=\fy_n-\fy_n^0$, with $\chi$ given in \eqref{def chi}. Then we have 
\EQ{
 \pt \|\fy_n/r\|_{L^2(D^<_{h_n})}\lec \|\na\fy_n^0\|_2,\pq \|\fy_n/r\|_{L^2(D^>_{h_n})}\lec\|\na\fy_n^1\|_2,
 \pr \|\na\fy_n\|_2^2=\|\na\fy_n^0\|_2^2+\|\na\fy_n^1\|_2^2+O(\|\na\fy_n\|_2\|\fy_n/r\|_{L^2(D_{h_n})}),
 \pr \|\fy_n\|_{2^*}^{2^*}=\|\fy_n^0\|_{2^*}^{2^*}+\|\fy_n^1\|_{2^*}^{2^*}+O(\|\na\fy_n\|_2^{2^*-2}\|\fy_n/r\|_{L^2(D_{h_n})}^2),
 }
where for the last error estimate, we used the radial Sobolev inequality 
\EQ{
 \|r^{d/2-1}\fy\|_\I \lec \|\na\fy\|_2 \pq(\fy\in\dot H^1_\rad).}
Then for large $n$ and $j=0,1$, we have $H(\fy_n^j)<H(W)-O(\e^2)$, and so, by the optimality of $W$ for the Sobolev inequality,
\EQ{
 K(\fy_n^j)\gec\e^2\|\na\fy_n^j\|_2^2\gec\e^4.}
which contradicts 
\EQ{
 o(1)=K(\fy_n)=K(\fy_n^0)+K(\fy_n^1)+o(1) \pq(n\to\I).}
Thus we obtain \eqref{no dichotomy}. Its right-hand side can hold only for a limited number $N(\nu)=O(\nu^{-2})$ of $h\in\Z$ for each $n$, since 
\EQ{
 \sum_{h\in\Z}\|\fy_n/r\|_{L^2(D_h)}^2\lec\|\na\fy_n\|_{2}^2 \lec 1.}
Hence we can rescale $\fy_n\mapsto\lam_n^{d/2-1}\fy_n(\lam_nx)$ so that for any $\e>0$ there are $j<k\in\Z$ such that for all $n$
\EQ{
 \|\fy_n/r\|_{L^2(r<2^j \cup r>2^k)}<\e,}
which controls the $L^{2^*}$ norm on the same region, via the radial Sobolev estimate as above. Since $\fy_n$ converges strongly  in $L^{2^*}(2^j<r<2^k)$ by
 the radial Sobolev, we conclude that the rescaled $\fy_n$ converges to some $\fy_\I$ in $L^{2^*}(\R^d)$. Since all the functional properties are preserved by the rescaling, we deduce 
\EQ{
 \|\na\fy_\I\|_2\le\|\na W\|_2,\pq K(\fy_\I)\le 0.}
The uniqueness of $W$ as the Sobolev maximizer implies that $\fy_\I\in \calS$, and then the norm convergence implies the strong convergence in $\dot H^1$. However, this implies that $\dist_{\dot H^1}(\fy_\I,\calS)\gec\de^2$, a final contradiction. 
\end{proof}

In order to analyze the behavior of $d_\calS^2(u)\sim E_\calS(u)$ and thereby also $K(u)$ close to the ground states, we will crucially employ the following ejection lemma. 

\begin{lem}\label{lem: ejection} 
There exists $\de_H\in(0,\de_E)$ with the following properties: Let $u$ be a solution on an open interval $I$ such that for some $t_0\in I$ 
\EQ{ \label{away from S}
 \de_0:=d_\calS(\vec u(t_0))\le \de_H, \pq E(\vec u)-J(W) \le \de_0^2/2,} 
and 
\EQ{ \label{outgoing}
  \p_t d_\calS(\vec u(t_0))\ge 0.}
Then for $t>t_0$ in $I$ and as long as $d_\calS(\vec u(t))\le\de_H$, $d_\calS(\vec u(t))$ is increasing, 
\EQ{ \label{eject est}
 \pt d_\calS(\vec u(t)) \sim -\sg\mu_\calS(u(t)) \sim e^{k(t-t_0)\lam(u(t_0))}\de_0, 
 \pr \sg K(u(t)) \gec (e^{k(t-t_0)\lam(u(t_0))}-C_*\LR{(t-t_0)\lam(u(t_0))})\de_0
 \pr |\lam(u(t))-\lam(u(t_0))| \lec (e^{k(t-t_0)\lam(u(t_0))}-1)\de_0\lam(u(t_0)),}
for some absolute constant $C_*>0$ and $\sg=\pm 1$ is fixed on the time interval. 
\end{lem}
\begin{proof} 
We will show that the hyperbolic mode $\mu_\calS$ grows exponentially, dominating the other modes. 
The main difficulty we encounter by comparison to~\cite{NakS} is that we need to pay attention to the evolution of the root mode, 
or equivalently the scaling parameter $\lam(u(t))$, which cannot be controlled by the energy or other conserved quantities. 
What saves us is that the evolution of $\lam(u(t))$ is slow enough that it can be ignored compared with the exponential growth. 
Without loss of generality we rescale to achieve
\EQ{
 \lam(u(t_0))=1,}
and we work first with this fixed scale. We may also assume that $u$ is close to $\calS$ at $t=t_0$, decomposing it by
\EQ{ \label{decompose at 1}
 u = W+v_1 = W+\mu_1(u)\rh+\ga_1.}
We prove exponential upper bounds by a bootstrap argument.  

{\bf{Bootstrap assumption}}: {\it We assume, for some large constant $M\gg 1$, and for $t\in I$ such that $M^2 \de_0 e^{k(t-t_0)}\ll 1$,}  
\EQ{ \label{bootstrap}
 \pt |\vec\mu_1(u(t))| \le M \de_0 e^{k(t-t_0)},
 \pr \|\vec\ga_1(t)\|_2 \le M \de_0 \LR{t-t_0} + M^3 \de_0^2 e^{2k(t-t_0)},}
which implies $\|\vec v_1(t)\|_2\lec M\de_0e^{k(t-t_0)}$. 
We will show that better bounds hold under the above assumption. Then by the time continuity, we obtain the above bound on any such time interval. 
We emphasize that in this argument we do not employ any dispersive estimates. 

In the following, we abbreviate $\mu_1(t)=\mu_1(u(t))$. Then $v_1=u-W$ solves 
\EQ{
 \ddot v_1 + \calL v_1 = N(v_1) \pt:= |W+v_1|^{2^*-2}(W+v_1) - W^{2^*-1} - (2^*-1)W^{2^*-2}v_1 
 \pr= O(W^{2^*-3}v_1^{2} + |v_1|^{2^*-1}), }
and so, the eigenmode solves 
\EQ{ \label{eq:mu1}
 (\p_t^2-k^2)\mu_1 &= \LR{N(v_1)|\rho}.} 
This leads to the integral equation 
\EQ{ \label{mu1 Duh}
 \mu_1(t) \pt= \mu_+(t) + \mu_-(t) 
 + \int_{t_0}^t\frac{\sinh(k(t-s))}{k}\LR{N(v_1)(s)|\rh}\,ds, 
}
where $\mu_\pm(t)$ denote the solutions of the linearized equation 
\EQ{
 \mu_\pm(t) := e^{\pm k(t-t_0)}\frac{1}{2}\bigl[1 \pm \frac{1}{k}\p_t\bigr]\mu_1(t_0).}
Our assumptions at time $t=t_0$ imply that $|\mu_{\pm}(t_0)|\lesssim \de_0 \ll 1$. Furthermore, we estimate  via the bootstrap assumptions 
\EQ{ \label{eq:m1int}
 \Bigl|\int_{t_0}^t e^{\pm k(t-s)}\LR{N(v_1)(s)|\rh}\,ds \Bigr| \pt\lec   \int_{t_0}^t  e^{k(t-s)}  \|N(v_1(s))\|_{L^{2d/(d+2)}} \, ds 
 \pr\lec \int_{t_0}^t  e^{k(t-s)} \|v_1(s)\|_{\dot H^1}^{2} ds \lesssim M^{2}  \de_0^2e^{2k(t-t_0)}.
}
It is immediate from this that 
\EQ{ \label{eq:bootstrmu1}
 |\vec\mu_1(t)|  &\lec \de_0 e^{k(t-t_0)} +   M^{2}\de_0^2  e^{2k(t-t_0)} \ll M\de_0e^{k(t-t_0)}, 
}
since the right-hand side is small. 
To bound the remainder $\ga_1$, we use the energy identity. 
Multiplying the equation of $\mu_1$ with its time derivative yields 
\EQ{
 \p_t[\dot\mu_1^2/2-k^2\mu_1^2-C_1(\mu_1\rh)]=\LR{N(v_1)-N(\mu_1\rh)|\dot\mu_1\rh}.}
Subtracting it from the energy of $v_1$ 
\EQ{
 E(\vec u)-J(W)=\frac{\dot\mu_1^2-k^2\mu_1+\|\dot \ga_1\|_2^2+\LR{\calL\ga_1|\ga_1}}{2}-C_1(v_1),}
we obtain
\EQ{ \label{ene v+}
 \p_t[\|\dot \ga_1\|_2^2/2+\LR{L\ga_1|\ga_1}/2-C_1(v_1)+C_1(\mu_1\rh)]=\LR{N(\mu_1\rh)-N(v_1)|\dot\mu_1\rh}.}
The nonlinear terms are estimated by H\"older and Sobolev (using that $v_1$ is small)
\EQ{
 \pt |C_1(v_1)-C_1(\mu_1\rh)| \lec \|\na\ga_1\|_2\|\na v_1\|_2^2, 
 \pr |\LR{N(v_1)-N(\mu_1\rh)|\dot\mu_1\rh}|\lec \|\na\ga_1\|_2\|\na v_1\|_2|\dot\mu_1|.}
Hence by time integration using the bootstrap bounds, one concludes that 
\EQ{ \label{bd vt}
 \|\dot \ga_1\|_2^2+\LR{\calL\ga_1|\ga_1} \lec \de_0^2+(M\de_0\LR{t-t_0}+M^3\de_0^2e^{2k(t-t_0)})M^2\de_0e^{2k(t-t_0)}.}
The orthogonality \eqref{lam orth} at $t=t_0$ implies that 
\EQ{
 \LR{\ga_1|\La^*\rh}=\int_{t_0}^t\LR{\dot\ga_1|\La^*\rh}dt.}
Hence we can estimate $\|\na\ga_1\|_2$ by using Lemma \ref{lem:non-neg spec} and \eqref{bd vt}. Thus we obtain  
\EQ{
 \pt\|\dot\ga_1\|_2 \lec \de_0 + M^{3/2}\de_0^{3/2}\LR{t-t_0}^{1/2}e^{k(t-t_0)}+M^{5/2}\de_0^2e^{2k(t-t_0)},
 \pr\|\na\ga_1\|_2 \lec \de_0\LR{t-t_0}+M^{3/2}\de_0^{3/2}\LR{t-t_0}^{1/2}e^{k(t-t_0)}+M^{5/2}\de_0^2e^{2k(t-t_0)},}
which is better by $O(M^{-1/2})\ll 1$ than the bootstrap assumption. 
This completes the bootstrap argument, whence the proof of~\eqref{bootstrap}. 
Henceforth, we shall regard $M$ as being an absolute constant and ignore it. 

Next we utilize the monotonicity assumption \eqref{outgoing} on $E_\calS(\vec u)$ in order to obtain a lower bound on $\mu_1$ in the same form. The technical difficulty we face here is that $E_\calS(\vec u)$ is defined with respect to the time-dependent scale $\lam(u(t))$, while the above estimates are at the fixed scale $1=\lam(u(t_0))$. The idea is that $E_\calS(u)$ should differ from $E_1(u)$ only by $O((t-t_0)^2)$ with a small multiple. To see this, we compare the two decompositions  
\EQ{ \label{eq:movingdecomp}
 u(t)\pt=W+v_1(t)=W + \mu_1(u(t))\rho  + \ga_1(t) 
 \pr=W_{\lam}+v_{\lam}(t) = W_{\lam} + \mu_{\lam}(u(t))\rho_\lam + \ga_\lam(t),}
where $\lam=\lam(u(t))$ is chosen according to \eqref{lam orth}. Then we have 
\EQ{ \label{eq:delicate}
  E_1(u) - E_\calS(u) = k^2[\mu_{1}(u)^2 - \mu_{\lam}(u)^2], 
}
as long as $u$ remains close to $\calS$. The right-hand side is estimated by 
\EQ{ 
 \mu_1(u)-\mu_\lam(u)\pt=\LR{v_1|\rh}-\LR{v_\lam|T_{1/\lam}^*\rh}
 \pr=\LR{v_1-v_\lam|\rh}+\LR{v_\lam|(1-T_{1/\lam}^*)\rh}
 \pn=O((\lam-1)^2),}
where we used that 
\EQ{
 \pt v_1-v_\lam = W_\lam - W= (\lam-1)\La W + O((\lam-1)^2)
 \pr (1-T_{1/\lam}^*)\rh_\lam = (\lam-1)\La^*\rh_\lam + O((\lam-1)^2),}
$\La W\perp\rh$ and $v_\lam\perp\La^*\rh_\lam$. 
On the other hand, using the upper bound \eqref{bootstrap}, we have
\EQ{
 |\LR{\La^*\rh|v_1(t)}|=|\LR{\La^*\rh|v_1(t) - v_1(t_0)}| 
 \pt= |\LR{\La^*\rh, \int_{t_0}^t \dot v_1(s)\,ds}|
 \pr\lesssim (e^{k(t-t_0)}-1)\de_0 \ll 1, 
}
hence the implicit function theorem implies that 
\EQ{
|\lambda(u(t)) -1| \lesssim (e^{k(t-t_0)}-1)\de_0, }
and so, 
\EQ{ \label{eq:delicatedifference}
 |E_1(u)-E_\calS(u)| \lec (e^{k(t-t_0)}-1)^2\de_0^3.}
This implies in particular that 
\EQ{
 \p_t E_1(\vec u(t_0)) = \p_t E_\calS(\vec u(t_0)) \ge 0,}
where the last inequality follows from the ``exiting assumption''~\eqref{outgoing}. 
From the energy conservation and the equation \eqref{eq:mu1} of $\mu_1$, we have 
\EQ{
 \pt\p_t E_1(\vec u(t)) = \p_t k^2\mu_1^2 = 2k^2\mu_1\dot\mu_1,}
hence $\p_t E_1(\vec u(t_0))\ge 0$ implies $\mu_+(t_0)\sim\mu_1(u(t_0))\sim \de_0$, and so via \eqref{mu1 Duh}, finally 
\EQ{
 |\mu_1(u(t))| \sim e^{k(t-t_0)}\de_0.}
By continuity, there is $\sg=\pm 1$ constant such that $\sg\mu_1(u(t))<0$. 
Expanding $K$ around $W$, and plugging the above estimates into this expansion yields  
\EQ{ \label{expand K}
 \sg K(u)\pt=-\sg(2^*-2)\LR{W^{2^*-1}|v}+O(\|\na v\|_2^2)
 \pr\gec \mu_1(u(t))-O(\|\vec\ga_1(t)\|_2)
 \gec (e^{k(t-t_0)}-C_*\LR{t-t_0})\de_0.}
To finish the proof of the lemma, it only remains to establish the monotonicity of $d_\calS$. At the fixed scale $1$, it is immediate from the equation that 
\EQ{
 \p_t^2 E_1(\vec u(t))=2k^2(\dot\mu_1^2+\mu_1\ddot\mu_1)
 \ge 2k^2\mu_1(k^2\mu_1+\LR{N(v)|\rh}) \gec e^{2k(t-t_0)}\de_0^2.}
Combining this with \eqref{eq:delicatedifference}, we infer that 
\EQ{
 E_\calS(\vec u(t)) \ge E_\calS(\vec u(t_0))(1+ c(t-t_0)^2),}
with some constant $c>0$. If $E_\calS(\vec u(t))$ becomes decreasing, or more precisely, $\p_tE_\calS(\vec u(t_1))=0$ at some $t_1>t_0$ before reaching $\de_H^2$, then we can apply the above argument backward in time from $t_1$ to concludes that $E_\calS(\vec u(t_0))> E_\calS(\vec u(t_1))$. However, this
 contradicts the above estimate. Hence $d_\calS(\vec u(t))$ is increasing, all the way until it reaches $\de_H$. 
\end{proof}

\section{The one-pass theorem} 

The key step in the proof of Theorem~\ref{thm: Main} consists of the following  assertion. 

\begin{thm}\label{thm: onepass} 
There exist
 $0<\e_*\ll \de_*\ll\de_H$ with the following properties:  Let $\vec u\in C(I;\HH)$ be a solution of \eqref{eq:H1critical} on an open interval $I$, satisfying for some $\e\in(0,\e_*]$, $\de\in(\sqrt 2\e,\de_*]$ and $T_1<T_2\in I$
\EQ{
 E(\vec u)\le J(W)+\e^2, \pq d_\calS(\vec u(T_1))<\de=d_\calS(\vec u(T_2)).}
Then   $d_\calS(\vec u(t))>\de$ for all $t>T_2$ in $I$. 
\end{thm}
\begin{proof} 
By increasing $T_1$ and decreasing $T_2$ if necessary, we may assume in addition that $\sqrt 2\e<d_\calS(\vec u(T_1))$ and $d_\calS(\vec u(t))$ is nondecreasing on $[T_1,T_2]$. Then Lemma \ref{lem: ejection} applies for all $t\in[T_1,T_2]$ and so $d_\calS(\vec u(t))$ is increasing for $t>T_1$ until it reaches $\de_H$. Arguing by contradiction, we assume that for some $t>T_2$ we have $d_\calS(\vec u(t))\le\de$. Such a $t$ can occur only away from $T_2$ (this will be made more precise shortly), and after $d_\calS(\vec u(t))$ has increased to size $\de_H\gg\de$. Moreover, by applying Lemma \ref{lem: ejection} backward in time, we can find $T_3>T_2$ such that $d_\calS(\vec u(t))$ decreases from $\de_H$ down to $\de$ as $t\nearrow T_3$, and so 
 that $d_\calS(\vec u(t))>\de$ for $T_2<t<T_3$. We may further assume 
\EQ{
 \lam(u(T_2))=1 \le \lam(u(T_3)),}
by rescaling and reversing time, if necessary. 

The theorem is now proved by deducing a contradiction from a localized virial identity, as in \cite{NakS}. Our argument differs from that in~\cite{NakS} in the following  two points:
\begin{enumerate}
\item The estimates in the hyperbolic regime incorporate the scaling changes. 
\item The degeneration $\|\na u(t)\|_2\ll 1$ is treated by the equipartition of energy. 
\end{enumerate}
(1) already appeared in \eqref{eject est}, which is essential in the critical case. (2) seems to be a more general argument than that used in~\cite{NakS}. 
The latter relies on the time oscillation at zero frequency as well as the subcriticality of the equation, neither of which is available for the critical wave equation.

Following~\cite{NakS}, introduce a space-time cutoff function 
\EQ{ \label{def w}
 w(t, x) = &\chi\Bigl(\frac{|x|}{t-T_2+\tau_2}\Bigr) \chi\Bigl(\frac{|x|}{T_3-t+\tau_3}\Bigr),}
where $\chi$ is the same cut-off function as in \eqref{def d_S}, and define\footnote{We are going to recycle this argument with $\ta_2\not=\ta_3$.} 
\EQ{
 \tau_2 = \tau_3 = \delta^{-1}.}
Then from the equation of $u$ we obtain the localized virial identity 
\EQ{ \label{eq:virial1}
\frac{d}{dt}V(t) = 2K(u(t)) + O(\Ext(t)), \pq V(t):=\LR{wu_t|(x\cdot\na+\na\cdot x)u},}
where the exterior free energy is denoted by 
\EQ{ \label{def Eext} 
 \Ext(t): = \frac{1}{2}\int_{|x|\ge R(t)}[|\nabla u|^2 + u_t^2]\,dx,\pq R(t):=\max(t-T_2+\ta_2,T_3-t+\ta_3),}
so that $\supp\p_{t,x}w\subset\{|x|\ge R(t)\}$.
By the finite speed of propagation, we have 
\EQ{ \label{ext energy bd}
 \sup_{T_2<t<T_3}\Ext(t) \lec \max_{j=2,3} \Ext(T_j)\lec\de,} 
where the last estimate follows from $d_\calS(\vec u(T_j))=\de$, $\lam(T_j)\ge 1$, $\ta=\de^{-1}$, and 
\EQ{
 \|\na W_\lam\|_{L^2(|x|>R)} \lec (R/\lam)^{1-d/2}, \pq 1-d/2\le -1/2.}
The left-hand inequality in \eqref{ext energy bd} is proved as follows. For each $T_j$, we can find a free solution $u^0_j$ so that $\p_{t,x}u^0_j(T_j,x)=\p_{t,x}u(T_j,x)$ on $|x|>R(T_j)=\ta$, and 
\EQ{
 \|\vec u^0_j\|_2^2\lec \Ext(T_j) \ll 1,} 
by a suitable extension to $|x|<\ta$. 
Since $\de\ll 1$, the small data wellposedness theory implies that there exists\footnote{Here we do not need the global Strichartz estimate or the scattering property, but the global existence follows from the local wellposedness combined with conservation of the small energy $E(\vec u_j)\sim\|\vec u_j\|_2^2$ as well as $K(u_j)\ge 0$.} a global solution $u_j$ of \eqref{eq:H1critical} with the same initial data as $u^0_j$ at $t=T_j$, which moreover satisfies
\EQ{
 \|\vec u_j\|_{L^\I_tL^2_x}\lec\|\vec u^0_j\|_{L^\I_tL^2_x} \lec \Ext(T_j)^{1/2} \ll 1.}
The propagation property of the linear wave together with the uniqueness for the nonlinear equation implies that $u^0_2=u$ for $|x|>R(t)$ and $T_2<t<(T_2+T_3)/2$, and $u^0_3=u$ for $|x|>R(t)$ and $(T_2+T_3)/2<t<T_3$. Thus we obtain \eqref{ext energy bd}, and so
\EQ{
 \dot V(t) = -2K(u(t)) + O(\de).}

On the other hand, the decay property of $W_\lam$ together with $d_\calS(\vec u(T_j))=\de$ and our choice of cut-off $\ta=\de^{-1}$ implies that 
\EQ{ \label{eq:virialbdry}
 |V(T_2)|+|V(T_3)| \lec  \de(1+\ta^{2-d/2})+\delta^2\ta \lec \de^{1/2},}
hence we have
\EQ{ \label{est virial}
  \left|\int_{T_2}^{T_3} [2K(u(t))- O(\de) ]\,dt\right| \lec \de^{1/2},}
which we are going to lead to a contradiction. 

Now we choose two parameters $0<\de_M\ll\de_H$ and $0<\nu\ll 1$, and set 
\EQ{ \label{choice of parameters}
 0<\de_*^{1/2}\ll\min(\de_M,\ka(\de_M),\nu^2), \pq 0<\e_*\ll\min(\de_*,\e_V(\de_M)),}
where $\ka(\cdot)$ and $\e_V(\cdot)$ are as in Lemma \ref{lem: variational}. 

First we consider the hyperbolic region in $[T_2,T_3]$. 
Let $\mm$ be the collection of local minimal points of $d_\calS(\vec u(t))$ in $[T_2,T_3]$, with the respective 
local minima less than $\de_M$. Since $\de\ll\de_M$, we have $T_2,T_3\in \mm$. For each $t_*\in \mm$, 
applying Lemma~\ref{lem: ejection} forward and/or backward in time, we obtain a subinterval $\hat I(t_*)\ni t_*$ of $[T_2,T_3]$ such that for $t\in \hat I(t_*)$   
\EQ{
 d_\calS(\vec u(t)) \sim e^{k|t-t_{*}|\lam(u(t_*))}d_\calS(\vec u(t_*)),}
and on the boundary $\p\hat I(t_*)$,  either $d_\calS(\vec u(t))=\de_H$, $t=T_2$ or $t=T_3$. 
In the latter cases, $d_\calS(\vec u(t))=\de_H$ at the other endpoint. 
Hence we have 
\EQ{ \label{bd hatI}
 |\hat I(t_*)|\gec \lam(u(t_*))^{-1}\log(\de_H/\de_M(\vec u(t_*))\ge \lam(u(t_*))^{-1}\log(\de_H/\de_M),}
and those intervals are mutually disjoint. Let\footnote{We chose the decomposition into $I_H$ and $I_V$ to maximize the use of the hyperbolic dynamics. One can also use the variational estimate in the overlapping region.}
\EQ{ \label{def IHV}
 I_H :=\bigcup_{t_*\in\mm}\hat I(t_*) \subset [T_2,T_3], \pq I_V:=[T_2,T_3]\setminus I_H.}
Then by the definition of $\mm$, we have $d_\calS(\vec u(t))\ge\de_M$ on $I_V$, and so Lemma \ref{lem: variational} implies 
\EQ{ \label{bd K on IV}
 \sg K(u(t)) \ge \min(\ka(\de_M),c\|\na u(t)\|_2^2),}
with $\sg=\pm 1$ constant on each connected component of $I_V$. 
Moreover, \eqref{eject est} implies that $K(u(t))$ has the same sign at the two endpoints for each internal $\hat I(t_*)$. 
Since $K(u(t))$ cannot change its sign while $t\in I_V$, we deduce that $\sg$ in \eqref{bd K on IV} for $t\in I_V$ and $\sg$ in \eqref{eject est} for $t\in I_H$ are the same constant sign on the whole $[T_2,T_3]$. 

Using the estimate on $K$ in \eqref{eject est}, we obtain for each $t_*\in\mm$, 
\EQ{ 
 \pt\sg\int_{\hat I(t_*)}[2K(u(t))-O(\de)]\,dt 
 \pr\gec \int_{\hat I(t_*)}(e^{k|t-t_*|\lam(u(t_*))}-2C_*\LR{(t-t_*)\lam(u(t_*))})d_\calS(u(t_*))\,dt 
 \gec \frac{\de_H}{\lam(u(t_*))},}
where the $O(\de)$ error was absorbed by the linearly growing factor. 
Moreover, the latter is absorbed by the exponentially growing factor after integration, since $d_\calS(\vec u(t_*))\le \de_M\ll\de_H$. Combining this and \eqref{bd hatI}, we infer that
\EQ{ \label{bd on IH} 
 -\sg\int_{\hat I(t_*)}[2K(u(t))-O(\de)]dt \gec \frac{\de_H}{\log(\de_H/\de_M)}|\hat I(t_*)| \gec \de_M|\hat I(t_*)|.}
Further, since $T_2\in\mm$ where $\lam(u(T_2))=1$, one has 
\EQ{
  |\hat I(T_2)| \gec \log(\de_H/\de_M) \gg 1.}

On $I_V$, we use the variational bound \eqref{bd K on IV}. 
If $\sg=-1$, the bound is uniform and 
\EQ{
 -K(u(t))\ge \ka(\de_M) \gg \de_*>\de.}
Hence 
\EQ{
 \sg\int_{I_V}[2K(u(t))-O(\de)]dt \gec \ka(\de_M)|I_V|.}
Combining it with the above estimate on $I_H$, we obtain  
\EQ{ 
 -\sg\int_{T_2}^{T_3}[2K(u(t))-O(\de)]\, dt \gec \de_M \gg \de_*^{1/2} > \de^{1/2},}
which contradicts \eqref{est virial}, concluding the proof in the case $\sg=-1$. 

If $\sg=+1$, then the lower bound degenerates as $\|\na u(t)\|_2\to 0$. This scenario can occur along some trajectory, since our equation is of the second order in time.  
We are going to show that it does not essentially affect the time integral by using energy equipartition. Decompose $I_V$ into 
\EQ{
 I_0:=\{t\in I_V \mid d_\calS(\vec u(t))<\nu\}, \pq I_1:=I_V\setminus I_0.}
The above argument implies that on $I_1$ we have 
\EQ{
 K(u(t)) \ge \min(\ka(\de_M),\nu^2) \gg \de_*>\de,}
and so
\EQ{ \label{bd on I1}
 \int_{I_1}[2K(u(t))-O(\de)]\,dt \gec \min(\ka(\de_M),\nu^2)|I_1|,}
whereas on $I_0$ we have $K(u(t))\sim\|\na u(t)\|_2^2$ and so
\EQ{ \label{bd on I0}
 \int_{I_0}[2K(u(t))-O(\de)]\, dt \gec \int_{I_0}\|\na u(t)\|_2^2\, dt-O(\de)|I_0|.}
In order to control this, we consider the energy equipartition with the same space-time cut-off as above for the virial identity: from the equation for $u$, 
\EQ{
 \p_t\LR{wu_t|u}\pt=\|\dot u(t)\|_2^2-K(u(t))+O(\Ext(t))
 \pr=\|\dot u(t)\|_2^2-K(u(t))+O(\de).}
Then in the same way as for \eqref{est virial}, we obtain 
\EQ{ \label{equipartition}
 \left|\int_{T_2}^{T_3}[\|\dot u(t)\|_2^2-K(u(t))+O(\de)]\, dt\right|\lec\de^{1/2}.}
On the other hand, \eqref{bd on IH}, \eqref{bd on I1} and \eqref{bd on I0} together with \eqref{est virial} imply 
\EQ{ \label{kinetic bd}
 \min(\de_M,\ka(\de_M),\nu^2)\int_{T_2}^{T_3}\|\na u(t)\|_2^2dt-O(\de)|I_0| \lec \de^{1/2},}
where we used the fact that $\vec u(t)$ is uniformly bounded in the case $\sg=+1$; this is obvious in the hyperbolic region $d_\calS(\vec u(t))<\de_H$, while in the exterior it follows from that $K(u(t))\ge 0$, since 
\EQ{ \label{ene bd K+}
 E(\vec u)-K(u(t))/2^*=\|\na u(t)\|_2^2/d+\|\dot u(t)\|_2^2/2.}
 Using \eqref{kinetic bd} and \eqref{equipartition} as well as \eqref{choice of parameters}, we deduce 
\EQ{
 \int_{T_2}^{T_3}[\|\dot u(t)\|_2^2+\|\na u(t)\|_2^2] dt \ll 1+\de^{1/2}|T_3-T_2|,}
which contradicts the energy conservation
\EQ{
 \int_{T_2}^{T_3}E(\vec u) dt =|T_3-T_2|E(\vec u)>|T_3-T_2|J(W)/2,}
since $|T_3-T_2|>|I_H|+|I_0|\gg 1$. This concludes the proof in the case $\sg=+1$. 
\end{proof}

The above result can be restated in terms of the sign functional as in the subcritical case \cite{NakS}. Let 
\EQ{
 \pt \HH_* = \{\vec\fy\in\HH \mid E(\vec\fy)\le J(W)+\e_*^2\},
 \pr \HH_X = \{\vec\fy\in\HH_* \mid E(\vec\fy)< J(W)+d_\calS^2(\vec\fy)/2\}.}
It is easy to see that $\HH_*\setminus \HH_X$ is a small neighborhood of $\vec \calS\cup-\vec \calS$. 
\begin{cor} \label{cor: onepass}
There exists a continuous function $\Sg:\HH_X\to\{\pm 1\}$ with the following properties. 
\begin{enumerate}
\item Every solution $u$ in $\HH_*$ can change $\Sg(\vec u(t))$ at most once. Moreover, it can enter or exit the region $d_\calS(\vec u)<\de_*$ at most once. 
\item The region $\Sg=+1$ is bounded in $\HH$, while the region $\Sg=-1$ is unbounded. 
\item If $\vec\fy\in\HH_X$ and $E(\vec\fy)\le J(W)+\e_V^2(d_\calS(\vec\fy))$, then $\Sg(\vec\fy)=\sign K(\fy)$, with the convention $\sign 0=+1$. 
\item If $\vec\fy\in\HH_X$ and $d_\calS(\vec\fy)\le \de_M$, then $\Sg(\vec\fy)=-\sign\mu_\calS(\fy)$. 
\end{enumerate}
\end{cor}
Note that $\HH_*\setminus\HH_X$ is included in $d_\calS<\de_*$, and that (3)--(4) completely determine $\Sg(\vec\fy)$, since we have chosen $\e_*<\e_V(\de_M)$. Moreover, $\Sg(\vec\fy)$ depends only on $\fy$. 
\begin{proof}
Since (3) and (4) are overdetermining $\Sg$, we need the consistency of the conditions. However, this is provided by the ejection lemma \ref{lem: ejection}, starting from any solution in the overlapping region, where $d_\calS<\de_M\ll\de_H$. The second estimate in \eqref{eject est} implies that the two definitions coincide at least at the endpoint of the ejection $d_\calS(\vec u(t))=\de_H$. Since both signs are invariant along the continuous trajectory $\vec u$, they must be the same all the way from the starting point. Thus $\Sg$ is well defined, and then (1) is the conclusion of the one-pass theorem \ref{thm: onepass}. The boundedness in (2) has been shown between \eqref{kinetic bd} and \eqref{ene bd K+}, while it is obvious that the $\Sg=-1$ region is unbounded, since it contains all $\vec\fy$ with negative energy. 
\end{proof}

It remains to determine the fate of the solutions in $\HH_*$ with $d_\calS\ge\de_*$. We will do this in the following two sections for $\Sg=\pm1$ , respectively. 

\section{Blow-up after ejection}
\begin{prop}
No solution can stay strongly continuous in $\HH_*$ with $\Sg=-1$ and $d_\calS\ge\de_*$ for all $t>0$. 
\end{prop}
\begin{proof}
Suppose towards a contradiction that there is a solution $u$ on $0<t<\I$ in $\HH_*$ with $\Sg(\vec u(t))=-1$ and $d_\calS(\vec u(t))\ge\de_*$. 
Here we use the identity for $|u|^2$, localized in the same way as for the virial identity. 

We may assume $E(\vec u)>J(W)$, since otherwise the conclusion follows from \cite{KM2,DKM1,DKM2}. 
We choose a time-dependent cut-off function and the localized $L^2$ norm
\EQ{
 w(t,x)=\chi(|x|/(t+\ta)), \pq y(t)=\LR{wu|u},}
for a fixed large $\ta>0$ to be determined later. Using that $\dot w\ge 0$, we have  
\EQ{
 \pt\dot y=\LR{\dot wu+2w\dot u|u} \ge 2\LR{w\dot u|u},}
and using the equation and Hardy's inequality, 
\EQ{ \label{ytt}
 \ddot y\pt=\LR{2w|\dot u^2-|\na u|^2+|u|^{2^*}}+\LR{\ddot w u|u}+\LR{4\dot w u|\dot u}+2\LR{u\na w|\na u}
 \pr=2(\|\dot u\|_2^2-K(u))+O(\Ext(t)),}
where 
\EQ{
 \Ext(t):=\int_{|x|>t+\ta} [|\dot u|^2+|\na u|^2]\,dx
 \lec \Ext(0) \ll \e_*,}
by the same argument as for \eqref{ext energy bd}, provided that we choose $\ta$ sufficiently large. 

In order to control the right-hand side of \eqref{ytt}, we follow the argument in the previous section, below \eqref{est virial}. Note that in the $\Sg=-1$ case, the contradiction assumption at $t=T_3$ was used only for the upper bound on $|V(T_3)-V(T_2)|$, and so the rest of the argument is still valid. 

Let $I=(T_2,\I)=I_H\cup I_V$ and $I_H=\bigcup_{t_*\in\mm}\hat I(t_*)$ as before, see \eqref{def IHV}. 
We have $-K(u(t))\gg \de_*\gg\e_*$ on the variational region $I_V$, while $\int_{\hat I(t_*)}-K(u(t))dt \gg \e_*|\hat I(t_*)|$ on each hyperbolic interval $\hat I(t_*)$. Hence $\dot y(t)\to\I$ and $y(t)\nearrow\I$ as $t\to\I$. Moreover, we can rewrite 
\EQ{ \label{ytt low bd}
 \|\dot u\|_2^2-K(u)\pt=(1+2^*/2)\|\dot u\|_2^2+(2^*-2)\|\na u\|_2^2-2^*E(\vec u).}
In the variational region $I_V$, using that $K(u(t))<0$ we have  
\EQ{ \label{ene up bd}
 E(\vec u)<J(W)+\e_*^2=\frac{2^*-2}{2^*}\|\na W\|_2^2+\e_*^2<\frac{2^*-2}{2^*}\|\na u\|_2^2+\e_*^2,}
which implies 
\EQ{
 \|\dot u\|_2^2-K(u)>(1+2^*/2)\|\dot u\|_2^2-2^*\e_*^2.}
Interpolating it with the other lower bound $\|\dot u\|_2^2+\de_*$ and using $\de_*\gg\e_*^2$, we get 
\EQ{ 
 \ddot y > 4(1+c)\|\dot u\|_2^2 + \e_*^2 \pq (t\in I_V),}
for some constant $c>0$ (say $1/d$). In the other region $I_H$, the last inequality of \eqref{ene up bd} may fail, but the smallness $|K(u(t))|\lec\de_H\ll 1$ allows us to replace it by 
\begin{lem}
For any nonzero $\fy\in\dot H^1$, we have 
\EQ{
 \|\na W\|_2^2 \le \|\na\fy\|_2^2 + (d/2-1)K(\fy) + O(K(\fy)^2/\|\na\fy\|_2^2).}
\end{lem}
\begin{proof}
Since $\fy\not=0$, there is a unique $\lam>0$ such that $K(\lam\fy)=0$, 
that is  
\EQ{ \label{sol la}
 \lam^{2^*-2}=\|\na\fy\|_2^2/\|\fy\|_{2^*}^{2^*}.}
Since $W$ is the Sobolev optimizer with $K(W)=0$, we have 
\EQ{
 \|\na W\|_2^2 \le \|\na\lam\fy\|_2^2.}
Inserting \eqref{sol la}, we obtain the desired conclusion after  Taylor expansion. 
\end{proof}

Since $\|\na u\|_2^2\sim\|\na W\|_2^2$ in the hyperbolic region $I_H$, we thus replace \eqref{ene up bd} with
\EQ{
 E(\vec u)<\frac{2^*-2}{2^*}\|\na u\|_2^2+\frac{d-2}{d}K(u)+O(K(u)^2+\e_*^2),}
and so from \eqref{ytt low bd} we obtain 
\EQ{
 \ddot y > 4(1+c)\|\dot u\|_2^2-2K(u)-O(K(u)^2+\e_*^2) \pq (t\in I_H).}
The leading term is bounded from below via Cauchy-Schwarz:
\EQ{
 4(1+c)\|\dot u\|_2^2\ge (1+c)\frac{|\dot y|^2}{y}.}
Hence  
\EQ{
 \pt \ddot y \ge (1+c)(\dot y)^2/y + \CAS{\e_*^2 &(t\in I_V) \\ -2K(u)-O(K(u)^2+\e_*^2) &(t\in I_H).}}
Hence $y$ is convex on $I_V$, while on each interval $\hat I(t_*)$ in $I_H$, we have in the same way as for \eqref{bd on IH}, 
\EQ{
 \int_{\hat I(t_*)} [-2K(u)-O(K(u)^2+\e_*^2)] dt \gec \de_M|\hat I(t_*)|,}
since $|K(u)|\lec\de_H\ll 1$. Moreover, if $d_\calS(\vec u(t))=\de_H$ at both ends of $\hat I(t_*)$ 
(which is the case except for the first interval), then the above integral on $\hat I(t_*)\cap(-\I,T)$ 
is positive for any $T$; indeed, the main contribution comes from the region where $d_\calS(\vec u(t))\sim\de_H$, and it is much bigger than the negative contribution. 
Therefore $\dot y\to\I$ as $t\to\I$. In particular, $\dot y>0$ and $y\nearrow\I$ for large $t\gg 1$. Since 
\EQ{
 \p_t y^{-c}=-cy^{-1-c}\dot y, \pq \p_t^2y^{-c} = -cy^{-1-c}[\ddot y-(1+c)(\dot y)^2/y],}
and $y^{-1-c}$ is decreasing for large $t$, the same logic as above implies that $\p_t y^{-c}$ does not become bigger in each $\hat I(t_*)$ than its value at the left end of the interval. Hence $\p_t y^{-c}<-a$ for some $a>0$ uniformly for large $t$, which leads to a blow-up by contradiction. 
\end{proof}

\section{Scattering after ejection}
In the other region $\Sg=+1$, we already know that all solutions are uniformly bounded in $\HH$, but that is not sufficient for the global existence of strongly continuous solutions
 in the critical case. Now we resort to the recent result by Duyckaerts-Kenig-Merle \cite{DKM1,DKM2} to preclude concentration (type II) blow-up. This is the only place where we have to restrict the dimensions\footnote{Strictly speaking, the long-time perturbation argument should be also modified for $d>6$ in the scattering proof of Proposition \ref{prop:Sg+ scatt}, but it is a minor issue. See \cite{CNLKG,ScatBlow} for the solution.} to $3$ or $5$. 
 
\begin{prop} \label{prop:Sg+ global}
No solution blows up in $\HH_X$ with $\Sg=+1$. 
\end{prop}
\begin{proof} 
First, the ejection lemma \ref{lem: ejection} precludes blow-up in the hyperbolic region, since the scaling parameter is a priori bounded during the ejection process, which is valid when reversing the time direction. Hence a blow-up may happen only when $d_\calS(\vec u(t))>\de_H$, where $K(u(t))\ge 0$ and so \eqref{ene bd K+} implies 
\EQ{
 \|\dot u(t)\|_2^2/2+\|\na u(t)\|_2^2/d < J(W)+\e_*^2 = \|\na W\|_2^2/d+\e_*^2.}
This allows us to employ the main result in \cite{DKM1,DKM2}, after reducing $\e_*$ if necessary. 
Suppose $u$ is a solution on $[0,T_+)$ in $\HH_X$ with $\Sg=+1$ and $d_\calS(\vec u(t))>\de_H$ with the blow-up time $T_+<\I$. 
According to their result, we can then write for $t$ sufficiently near $T_+$
\EQ{ 
 \vec u(t) = \vec W_{\lam(t)} + \vec\fy + o(1) \pq\IN\HH,}
for some $0<\lam(t)\to 0$ and some fixed $\vec\fy\in\HH$. 
It is then easily checked that as $t \to T_+-0$ we have 
\EQ{ \label{eq:Kv}
 K(u(t)) = K(W_{\lambda(t)}) + K(\fy) + o(1) = K(\fy) +o(1),}
from which we infer in particular that $K(\fy)\ge 0$. Similarly, we obtain
\EQ{ \label{eq:Ev}
 J(W)+\e_*^2> E(\vec u) = J(W) + E(\vec\fy),}
which implies via \eqref{ene bd K+} and $K(\fy)\ge 0$,  
\EQ{
 \|\dot\fy\|_2^2/2 + \|\na\fy\|_2^2/d < \e_*^2.}
This however contradicts $d_\calS(\vec u(t))>\de_H \gg \e_*$ near $T_+$.  
\end{proof}

Next we employ the Kenig-Merle scheme from \cite{KM1,KM2} to improve the above result. The one-pass theorem will be incorporated in the same way as in the subcritical case \cite{NakS}. Extinction of the critical element requires a little extra work due to the possibility of concentration, which will be however reduced to the above proposition. 
\begin{prop} \label{prop:Sg+ scatt}
Every solution staying in $\HH_X$ with $\Sg=+1$ and $d_\calS\ge\de_*$ for $t>0$ scatters to $0$ as $t\to+\I$ with uniformly bounded Strichartz norms on $[0,\I)$. 
\end{prop}
The restriction $d_\calS\ge\de_*$ is essential for the uniform Strichartz bound, since the latter does not hold for all scattering solutions, even for $E(\vec u)<J(W)$. 
\begin{proof}
We argue by contradiction. Let $u_n$ be solutions on $[0,\I)$ in $\HH_X$ satisfying 
\EQ{\label{eq:unseq}
 \pt E(\vec u_n) \to E_{*}\le J(W)+\e_*^2,\quad \|u_{n}\|_{L^q_{t,x}(0,\I)} \to \I, 
 \pr d_S(\vec u_n(t))\ge \de_*, \pq \Sg(\vec u_n(t))=+1, \pq (t>0)
}
where we choose $q=2(d+1)/(d-2)$ so that $L^q_{t,x}$ is an admissible Strichartz norm for the wave equation on $\R^d$. 
Henceforth, $X(I)$ denotes the restriction to $I\times\R^d$ of the Banach function space $X$ on $\R\times\R^d$. It is well-known that $L^q_{t,x}$ and the energy norm are sufficient to control all the other Strichartz norms, such as $L^p_t \dot B^{1/2}_{p,2}$ with $p=2(d+1)/(d-1)$, as well as the nonlinear term in some dual admissible norm such as in $L^{p'}_t \dot B^{1/2}_{p',2}$ (see, for example, \cite{GSV}). 

We may assume that $E_*$ is the minimum for the above property. 
Following the Kenig-Merle argument, the proof consists of two parts: construction and exclusion of a critical element. 

\medskip\noindent{\bf Part I}: Construction of a critical element.

Assuming the existence of \eqref{eq:unseq}, we are going to show that there is a critical element $u_*$, that is a solution on $[0,\I)$ in $\HH_X$ satisfying 
\EQ{ \label{def crit}
 \pt E(\vec u_*) = E_*,\pq \|u_*\|_{L^q_{t,x}(0,\I)}=\I,
 \pq d_S(\vec u_*(t))\ge \de_*, \pq \Sg(\vec u_*(t))=+1,}
and that its trajectory is precompact modulo dilations in $\HH$. 

If $d_\calS(\vec u_n(0))<\de_H$, then by the ejection lemma \ref{lem: ejection}, we have $d_\calS(\vec u_n(t))\ge \de_H$ at some later $t>0$. Since the Strichartz norm on the ejection time interval is uniformly bounded, we may translate each $u_n$ so that 
\EQ{
 d_\calS(\vec u_n(0))\ge\de_H,}
without losing \eqref{eq:unseq}. The translation time is bounded by $\lam(u_n(0))$: the scaling at $t=0$, which remains the same order after the translation. 

Since we chose $\e_*\ll\e_V(\de_H)$, Lemma \ref{lem: variational} implies 
\EQ{\label{eq:Klower}
 K(u_{n}(0)) \ge \min(\kappa(\delta_H),c\|\nabla u_{n}(0)\|_{2}^{2}).}
Now apply\footnote{In what follows, we will pass to subsequences without any further mention. Also note that Merle, Vega independently obtained a decomposition of this type for NLS, \cite{MeV}.} the Bahouri-Gerard decomposition from~\cite{BaG}, see also Lemma~4.3 in \cite{KM2}, to $\{\vec u_n(0)\}_{n\ge1}$. Let $U(t)$ denotes the free wave propagator. We conclude that there exist $\lam^j_n>0$, $t^j_n\in\R$, $\vec\fy^j\in\HH$ and free waves $w^J_n$ such that for any $J\ge1$
\EQ{ \label{eq:BG}
  \pt U(t)\vec u_{n}(0) = \sum_{j=1}^{J} \vec V^j_n(t) + \vec w^J_n(t), 
  \pq \vec V^j_n(t):=U(t+t^j_n)T^j_n\vec\fy^j,}
where $T^j_n$ is the operator defined by $T^j_n f :=(\lam^j_n)^{d/2} f(\lam^j_n x)$, such that 
\EQ{ \label{param orth}
 \pt |\log(\lam^j_n/\lam^k_n)|+|t^j_n-t^k_n|/\lam^j_n \to\I }
for each $j\not=k$, 
\EQ{ \label{ene orth}
 \pt \lim_{n\to\I} \Bigl[ \|\vec u_{n}(0) \|_2^{2} -\sum_{j=1}^{J}  \|\vec V^j_n(0)\|_2^{2} - \|\vec w^J_n(0)\|_{2}^{2} \Bigr]=0, 
 \pr \lim_{n\to\I} \Bigl[ E(\vec u_{n}(0)) - \sum_{j=1}^{J} E(\vec V^j_n(0)) - E(\vec w^J_n(0)) \Bigr] =0}
for each $J$, and 
\EQ{ \label{Str vanish}
  \lim_{J\to\I}\limsup_{n\to\I}\|w^J_n\|_{L^\I_t L^{2^*}_x(\R)\cap L^q_{t,x}(\R)}=0.} 
The last property applies to any other non-sharp Strichartz norm by interpolation, since those free waves are all uniformly bounded. 

First we check that all components retain $K\ge 0$ at $t=0$. Using \eqref{ene bd K+}, we get 
\EQ{
 E(\vec u_n)  - \frac{1}{2^*} K(u_{n}(0)) &\ge \frac 1d\|\vec u_n(0)\|_2^2 
 = \sum_{j=1}^{J} \frac 1d \|\vec V^j_n(0)\|_{2}^{2} + \frac 1d \|\vec w^J_n(0)\|_{2}^{2} +o(1).}
Hence if $\|\na u_n(0)\|_2^2\lec\e_*^2$, then $\|\na V^j_n(0)\|_2^2\lec\e_*^2\ll 1$, and so $K(V^j_n(0))\ge 0$. Otherwise, the lower bound in \eqref{eq:Klower} is much bigger than $\e_*^2$, so for large $n$, we get from the above inequality 
\EQ{
  H(V^j_n(0)) <  J(W),}
which implies $K(V^j_n(0)) \ge 0$, by the variational property of $W$. The same argument implies $K(w^J_n(0))\ge 0$ as well. Thus, each component has non-negative energy $E$. We may assume that $j=1$ gives the maximum among $E(\vec V^j_n(0))$, and so 
\EQ{ \label{LP_2 ene bd}
 E(\vec V^j_n(0))<\frac{2}{3}J(W),\pq (j>1).}

Now let $U^j$ be the nonlinear profile associated with $V^j_n$, that is the nonlinear solution satisfying as $n\to\I$, 
\EQ{ \label{eq:UjVj}
 \|\vec U^j(s^j_n)-U(s^j_n)\vec \fy^j\|_2 \to 0, \pq s^j_n:=\lam^j_n t^j_n,}
defined uniquely around $t=s^j_\I:=\lim_{n\to\I}s^j_n$, such that 
\EQ{
 \|\vec U^j_n(0)-\vec V^j_n(0)\|_2\to 0\pq \vec U^j_n(t):=(T^j_n\vec U^j)(\lam^j_n(t+t^j_n)).}
By the scaling invariance of the equation, each $U^j_n$ is also a solution, defined locally around $t=0$. Hence the above property of $\vec V^j_n(0)$ is transferred to $U^j_n$: 
\EQ{
 \pt K(U^j_n(0))\ge 0, \pq 0\le E(\vec U^j_n)=E(\vec U^j) \sim \|\vec U^j_n(0)\|_2^2,
 \pr \sum_{j=1}^J E(U^j) \lec J(W), \pq \sup_{j>1}E(\vec U^j)\le \frac{2}{3}J(W),}
and so, by \cite{KM2}, each $U^j$ for $j>1$ exists globally and scatters with 
\EQ{
 \sum_{j=2}^J \|U^j\|_{L^q_{t,x}(\R)}^2 \lec 1.}
Note that only a bounded number of profiles can escape from the small energy scattering theory, where all Strichartz norms are bounded by the energy norm. 

Now assume the same for $U^1$ and thus for all $j\ge1$, which is the case if $E(U^1)<J(W)$. 
Then from the long-time perturbation theory, cf.~Theorem~2.20 in~\cite{KM2}, 
one obtains the {\em nonlinear profile decomposition}  for the solutions $u_{n}(t)$, provided $J$ is large and fixed,
and $n\ge n_0(J)$ is sufficiently large: 
\EQ{\label{eq:nonlinearBG}
 \pt u_{n} = \sum_{j=1}^{J} U^j_n+ w^J_n  + R^J_n,
 \pq \lim_{J\to\I}\limsup_{n\to\I} \|\vec R^J_n\|_{(L^\I_t\HH\cap L^q_{t,x})(\R)}= 0, }  
which implies $u_n$ is bounded in $L^q_{t,x}$, contradicting \eqref{eq:unseq}. Thus we have obtained 
\EQ{
 \| U^1\|_{L^q_{t,x}(\R)}=\I, \pq J(W)\le E(U^1)\le E_*, \pq \sum_{j=2}^JE(U^j)+\|\vec w^J_n\|_2^2 \lec \e_*^2.}

We now distinguish three cases (a)--(c) by means of $s^1_\I=\lim_{n\to\I}\lam^1_nt^1_n$: 

\medskip\noindent{\bf (a)} $s^1_\I=\I$. Then by definition \eqref{eq:UjVj}, $U^1$ is a local solution around $t=\I$ with finite Strichartz norms, and 
\EQ{
 \|U^1_n\|_{L^q_{t,x}(0,\I)}=\|U^1\|_{L^q_{t,x}(s^1_n,\I)}\to 0.}
Hence we can use the long-time perturbation argument on $(0,\I)$, which gives a contradiction via~\eqref{eq:nonlinearBG} as above. 

\medskip \noindent{\bf (b)} $s_{1,\I}=-\I$.  In this case $U^1$ scatters at $t=-\I$ by definition. Let $I=(-\I, T_{+})$ be the maximal interval of existence of $U^1$.  

If $d_\calS(U^1(t))>\de_*/2$ for all $t<T_+$, then $U^1$ remains in $\HH_X$ with $\Sg=+1$ from $t=-\I$. Hence $T_+=\I$ by Proposition \ref{prop:Sg+ global}, and $\|U^1\|_{L^q_{t,x}(0,\I)}=\I$. Moreover, the one-pass theorem \ref{thm: onepass} together with the ejection lemma \ref{lem: ejection} implies that $d_\calS(U^1(t))\ge\de_*$ for large $t$. Hence $U^1$ is a critical element after some time translation. 

Otherwise, $d_{\calS}(U^1(t_{*})) = \de_*/2$ at some minimal $t_{*}<T_{+}$, until which $U^1$ remains in $\HH_X$ with $\Sg=+1$, and $\| U^1 \|_{L^q_{t_x}(-\I,t_{*})} <\I$. 
Hence one can apply the nonlinear profile decomposition on the interval $\lam^1_n(t+t^1_n)\le t_*$, which yields in particular 
\EQ{\label{eq:dSest}
 d_{\calS}(\vec u_{n}((t_{*}-s^1_n)/\lam^1_n)) 
 \le d_{\calS}(\vec U_{1}(t_{*})) + O(\e_*) + o(1) \le \frac{2}{3}\de_* +o(1),}
as $n\to\I$, provided $J$ is large enough.  However, since $t_{*}-s^1_n\to\I$, this contradicts our assumption $\inf_{t\ge 0}d_{\calS}(\vec u_{n}(t))\ge \de_*$.  
 To obtain the $O(\eps_{0})$-term in~\eqref{eq:dSest}, one uses the bound, valid for $J$ large and all $n\ge n_{0}$, 
\EQ{
 \sup_{\lam^1_n(t+t_n) \le t_{*}} \|\vec R^J_n(t) \|_2 \lec \eps_*}
which follows from the main estimate of Theorem~2.20 in~\cite{KM2}. 

\medskip \noindent{\bf (c)} $s^1_\I\in\R$. 
Let $(T_-,T_+)\ni s^1_\I$ be the maximal interval of existence for $U^1$. 
We know that $K(U^1(s^1_\I))\ge0$. 
Moreover, by the same perturbative arguments as above, the nonlinear profile decomposition~\eqref{eq:nonlinearBG} holds on $(T_-,T_+)/\lam^1_n-t^1_n$. 
Thus, as in the case (b), we deduce from $\inf_{t\ge 0}d_{\calS}(\vec u_{n}(t))\ge\delta_*$ that 
\EQ{
 \inf_{s^1_\I\le t<T_+} d_{\calS}(\V U^1(t)) \ge \de_*/2.}
Then the same argument as in (b) implies that $T_+=\I$ and $U^1$ is a critical element after time translation, provided that $\|U^1\|_{L^q_{t,x}(s^1_\I,\I)}=\I$. Otherwise $U^1$ scatters and the nonlinear profile decomposition holds on $[0,\I)$, contradicting \eqref{eq:unseq}. 

\medskip
Thus we arrive at the conclusion that $s^1_\I<\I$ and $U^1$ is a critical element after time translation. 
This implies $E(U^1)=E_*$ by the minimality, which extinguishes the other profiles $U^j$ ($j>1$) as well as the remainder $w^J_n$ as $n\to\I$, through the nonlinear energy decomposition. 

Having a critical element $u_*$, we apply the above argument to the sequence 
\EQ{
 u_n(t)=u_*(t-t_n), \pq t_n\to\I.}
The vanishing of all but one profile implies that for some continuous $\lam(t)>0$ 
\EQ{\label{eq:compact}
 \{\lam(t)^{-d/2}\vec u_*(t,x/\lam(t))\}_{t\ge 0} \subset \HH}
is precompact, concluding the first part of the proof. 

\medskip\noindent{\bf Part II}: Exclusion of a critical element. 

Let $u_*$ be a critical element \eqref{def crit}, hence 
\EQ{
 \vec w_*(t):=\varrho(t)^{d/2}\vec u_*(t,\varrho(t)x), \pq \varrho(t):=1/\lam(t)} 
for $t\ge 0$ is precompact in $\HH$. We proceed in three steps. 

\medskip\noindent{\bf Step 1}: $\limsup_{t\to\infty}\varrho(t)/t < \infty$.
To see this, note that by finite propagation speed, we have 
\EQ{
\lim_{R\to\infty}\sup_{t\geq 0} \|\V u_*(t)\|_{L^2(|x|> t+R)} = 0,}
whence we have 
\EQ{
\lim_{R\to\infty}\sup_{t\geq 0} \|\V w_*(t)\|_{L^2(|x|> (t+R)/\varrho(t))} = 0.}
If for some sequence of times $\{s_n\}_{n\geq 1}$ we had $\varrho(s_n)/s_n\to 0$, then by pre-compactness of $\{\V w_*(t)\}_{t\geq 0}$, we get $\|\vec w_*(s_n)\|_{L^2}\to 0$, whence also $\|\vec u_*(s_n)\|_{L^2}\to 0$, which would force $E_* = 0$, a contradiction. 

\medskip\noindent{\bf Step 2}: $\liminf_{t\to \infty}\varrho(t)/t >0$. 
This follows from the localized virial identity \eqref{eq:virial1} as in the proof of Theorem \ref{thm: onepass}. By the precompactness, there is $R>0$, depending on $u_*$, such that for all $t\ge 0$
\EQ{
 \int_{|x|>R\varrho(t)} |\dot u_*|^2+|\na u_*|^2 dx < \de.}
Suppose for contradiction that $\liminf_{t\to\I}\varrho(t)/t= 0$. 
Choose $T_3\gg T_2\gg 1$ and $\ta_2,\ta_3>0$ such that 
\EQ{
 \varrho(T_j)\ll \de T_j/R, \pq \ta_j=R\varrho(T_j).}
Then we have 
\EQ{
 \pt |\LR{wu_t|x\cdot\na u+\na\cdot xu}| + |\LR{wu_t|u}| \lec R\varrho(T_j) \ll \de_*T_j\pq (t=T_j,\ j=2,3),
 \pr \sup_{T_2<t<T_3}\Ext(t) \lec \max_{t=T_2,T_3}\Ext(t) < \de,}
where $w$ and $\Ext$ are as in \eqref{def w} and \eqref{def Eext}. Then we have in place of \eqref{equipartition}--\eqref{kinetic bd}, 
\EQ{
 \pt\int_{T_2}^{T_3}[\|\dot u\|_2^2-K(u(t))+O(\de_*)]dt \ll \de_*T_3,
 \pr\int_{T_2}^{T_3}[\de_*^{1/2}\|\na u(t)\|_2^2-O(\de_*)]dt \ll \de_*T_3,}
which leads to 
\EQ{
 |T_3-T_2|J(W)\le \int_{T_2}^{T_3}E(u)dt \ll \de_*^{1/2}T_3,}
a contradiction. Here again we assumed $E(u)\ge J(W)$ since in the other case one can easily get a simpler bound, as was done in \cite{KM1}.  

\medskip\noindent{\bf Step 3}: Construction of a blow up solution via re-scaling $u_*$. Pick a sequence $s_n\to \infty$ with 
$\lim_{n\to\infty}\varrho(s_n)/s_n = c\in (0, \infty)$, as well as $\V w_*(s_n)\to \exists\vec\fy$ in $L^2$. Define a sequence of solutions 
\EQ{
 u_n(t, x): = s_n^{d/2-1}u_*(s_n t, s_n x)}
whence we have $\V u_n(1)\to c^{-d/2}\vec\fy(x/c)$ in $L^2$. 

The above two steps imply that $\V u_n$ is precompact in $C([\ta,1];L^2)$ for any $0<\ta<1$, and so, after passing to a subsequence, it converges to some $\V u_\I$ in $C((0,1];L^2)$. By the local wellposedness theory, it has finite Strichartz norms locally in time, and so $u_\I$ is the unique strong solution on $(0,1]$ with the initial condition $\V u_\I(1)=\V\fy$. 
Clearly we also have $d_\calS(\V u_\infty(t))\geq \delta_*$ and $\Sg(\V u_\I(t))=+1$ for $0<t\le 1$. 

We now show that $u_\infty$ is a solution blowing up at $t=0$, which contradicts Proposition \ref{prop:Sg+ global}. 
The fact that $u_\infty$ blows up at $t=0$ follows from 

 {\bf Claim}: $u_\infty(t, x) = 0$ on $|x|>t$. 
To see this, pick $0<\eps\ll 1$ arbitrary, let $m$ large enough such that $\|\V w_*(s_m) - \V\fy\|_{L^2} \ll\eps$ and further pick $R>0$ such that $\|\V\fy\|_{L^2(|x|>R)}\ll \eps$. 
Then for $n>m$, we have 
\EQ{
 \|\V u_n(s_m/s_n)\|_{L^2(|x| > R\varrho(s_m)/s_n)} = \|w_*(s_m)\|_{L^2(|x|>R)}\ll \eps.}
From this and the finite propagation speed, we deduce that for $s_m/s_n\le t\le 1$ 
\EQ{ 
 \|\V u_n(t)\|_{L^2(|x|>R\varrho(s_m)/s_n+t-s_m/s_n)}\ll \eps.}
Letting $n\to \infty$, we infer that for $0<t\le 1$
\EQ{
 \|\V u_n(t))\|_{L^2(|x|>t)}\ll \eps.}
Since $\eps>0$ is arbitrary, this implies that $u_\infty$ is supported on $|x|\le t$, as claimed. This completes the proof of Proposition \ref{prop:Sg+ scatt}. 
\end{proof}

In order to complete the proof of Theorem~\ref{thm: Main}, we now exhibit open data sets at time $t=0$ such that we have blow up/scattering at $t=\pm\infty$, four possibilities in all. For this, we use the representation 
\EQ{ 
 u = W + v_1 = W + \mu_1(u)\rh + \ga_1,}
used in the proof of Lemma ~\ref{lem: ejection}, see \eqref{decompose at 1}. 
We pick data of the form 
\EQ{ 
 u(0) = W + a \rho + f,\pq \dot{u}(0) = b \rho + g,}
for some $a,b\in\R$, $f\in\dot H^1$ and $g\in L^2$ radial, with the conditions 
\EQ{\label{eq:conditions}
 \|\na f\|_{2} + \|g\|_{2} \ll |a|+|b| \ll \de_*.}
It then follows from the same argument as below \eqref{mu1 Duh} that we have 
\EQ{ 
 \pt \mu_1(t) = e^{k t}\mu_+ + e^{-k t}\mu_- +O(e^{2k|t|}(a^2+b^2)),
 \pr \|\V\ga_1(t)\|_2 \lec \LR{t}(|a|+|b|)+e^{2k|t|}(a^2+b^2),}
as long as $e^{k |t|}(|a|+ |b|)\lec \delta_H$, where $\delta_H$ is as in Lemma~\ref{lem: ejection}, and further 
\EQ{
 \mu_+ := \frac{1}{2}\left(a + \frac{1}{k} b\right),\pq 
 \mu_- := \frac{1}{2}\left(a - \frac{1}{k} b\right).}
Using the expansion of $K$ in \eqref{expand K} as well, it is now easy to see that under the conditions \eqref{eq:conditions} we obtain 4 disjoint open sets, depending on the signs of $a$ and $b$, such that $K(u)\lessgtr 0$ at the ejection times, i.e.~ the endpoints of the time interval around $0$ where $d_\calS(\vec u)\le\de_H$. This completes the proof of Theorem~\ref{thm:  Main}. 
\qedsymbol

\end{document}